\documentclass{article}

\usepackage{diagrams,amsbsy,amssymb,amsfonts,amsthm,hyperref}

\theoremstyle{plain}
\newtheorem{thm}{Theorem}[section]
\newtheorem{prop}[thm]{Proposition}
\newtheorem{lem}[thm]{Lemma}
\newtheorem{cor}[thm]{Corollary}
\newtheorem{rem}[thm]{Remark}
\newtheorem{defn}[thm]{Definition}

\newcommand{\LP}{L_p}
\newcommand{\X}{\mathcal{X}}
\newcommand{\decomp}{\mathcal{D}}
\newcommand{\LS}{L}
\newcommand{\RS}{R}
\newcommand{\LSo}{L_{\mathcal{D}_1}}
\newcommand{\RSo}{R_{\mathcal{D}_1}}
\newcommand{\LSt}{L_{\mathcal{D}_2}}
\newcommand{\RSt}{R_{\mathcal{D}_2}}

\newcommand{\opX}{\mathcal{L}(\X)}
\newcommand{\AD}{\mathcal{A}(\decomp)}
\newcommand{\TD}{T_{\decomp}}
\newcommand{\N}{\mathbb{N}}
\newcommand{\les}{\sigma_{l.e.}(T)}
\newcommand{\linf}{\ell_{\infty}}

\begin{document}

\begin{center}{\Huge Commutators on $\ell_1$}
\end{center}

\begin{center}{\Large Detelin Dosev \footnote{Research supported in part by NSF grant DMS-0503688}}
\end{center}

\begin{abstract}
The main result is that the commutators on $\ell_1$ are the operators not of the form $\lambda I + K$ with $\lambda\neq 0$ and $K$ compact.
 We generalize Apostol's technique  (1972, Rev. Roum. Math. Appl. 17, 1513 - 1534) to obtain this result and use this generalization to 
 obtain partial results about the commutators on spaces 
$\X$ which can be represented as $\displaystyle \X\simeq \left ( \bigoplus_{i=0}^{\infty} \X\right)_{p}$ for some $1\leq p<\infty$ or $p=0$. 
 In particular, it is shown that every compact operator on $L_1$ is a commutator. A characterization of the commutators on  
$\ell_{p_1}\oplus\ell_{p_2}\oplus\cdots\oplus\ell_{p_n}$ is given. We also show that strictly singular operators on $\linf$ are commutators.
\end{abstract}

\section{Introduction}
\label{intro}
The commutator of two elements $A$ and $B$ in a Banach algebra is given by
$$
[A,B] = AB - BA.
$$
A natural problem that arises in the study of  derivations on a Banach algebra is to classify the commutators in the algebra.
The first major contribution was due to Wintner(\cite{Wintner}), who proved that the identity in a unital Banach algebra is not a commutator.
This immediately implies that no non-zero multiple of the identity is a commutator and with no effort one can also obtain that no operator of the form 
$\lambda I + K$, where $K$  is a compact operator and $\lambda\neq 0$, is a commutator in the Banach algebra $\opX$ of all bounded linear operators
on the Banach space $\X$. The latter fact can be easily seen just by observing that the 
quotient algebra $\opX / K(\X)$ ($K(\X)$ is the space of compact operators on $\X$) also satisfies  the conditions of the Wintner's theorem. 
Let us note also that instead of considering the ideal of compact operators one can consider any proper ideal in  $\opX$. 
The observations we made are valid in this case as well.\\
For a Banach space $\X$ for which there is a unique maximal proper ideal in $\opX$  (which is the case for the spaces $\LP$ and $\ell_p$ for 
$1\leq p \leq \infty$ and $c_0$) 
one can hope to obtain a complete classification of the commutators on the space.  The natural conjecture is that the only operators on $\X$ that 
are  not commutators are the ones of the form $\lambda I + K$, where $K$  belongs to the unique maximal ideal in $\opX$ and $\lambda\neq 0$. 
In 1965 Brown and Pearcy (\cite{BrownPearcy}) made a breakthrough in this direction by proving that this is in fact a classification of the commutators
on a Hilbert space. Note that if $\X = \ell_p$ ($1\leq p<\infty$) or $\X = c_0$, the ideal of compact operators $K(\X)$ is the largest proper ideal 
in $\mathcal{L}(\X)$(\cite[Theorem 6.2]{Whitley}). Whitley's proof actually shows that the ideal of strictly singular operators is the largest ideal 
in the aforementioned spaces, but as he points out, a result of Feldman, Gohberg and Markus (\cite{GM})
shows that the compact operators are in fact the only closed proper ideal in $L(\X)$ for $\X = \ell_p$ ($1\leq p<\infty$) or $\X = c_0$.
In 1972, Apostol (\cite{Apostol_lp}) proved that every non-commutator on the space $\ell_p$ for $1<p<\infty$ is of the form 
$\lambda I + K$, where $K$ is compact and $\lambda\neq 0$. One year later he proved that the same classification holds in the case of $\X = c_0$ 
(\cite{Apostol_c0}).  While Apostol's approach in (\cite{Apostol_lp}) gave some information about the 
commutators in $\mathcal{L}(\ell_1)$, he was unable to give a complete characterization. His proof uses the fact that the unit vector basis in 
$\ell_p$ for $1<p<\infty$ is shrinking and this does not hold for $\ell_1$. We overcome this obstacle by using the structure of the 
infinite dimensional subspaces of $\ell_1$ rather than just the properties of the basis. \\
In the Section 2 we study spaces $\X$ for which $\displaystyle \X\simeq \left ( \bigoplus_{i=0}^{\infty} \X\right)_{p}$,
$1\leq p<\infty$ or $p=0$, where for $p\geq 1$
$$
\left( \bigoplus_{i=0}^{\infty} \X\right)_{p} := \{f=(f_1, f_{­2},\ldots)\,:\, f_i\in\X,\,i=1,2,\ldots, \,\|f\|^p=\sum_{i=0}^{\infty}\|f_i\|^{p}<\infty \}
$$
(for $p=0$, $\displaystyle \|f\| = \sup_{0\leq i<\infty}\|f_i\|$), and we generally assume
 that $\opX$ has a largest proper ideal. The notion of a decomposition of $\X$ will be introduced and it will be shown how it can be used
to obtain results about commutators on these spaces. In Section 3 we show that the compact operators on $\X$, where $\X$ admits the decomposition
$\displaystyle \X\simeq \left ( \bigoplus_{i=0}^{\infty} \X\right)_{p}$ (in the case $p=1$ we  will assume 
that $\X = L_1$ or $\X = \ell_1$), are commutators and as a corollary we will get that an operator that has a 
compact restriction to a sufficiently large subspace of $\X$ is also a commutator. Section 3 contains our main result - Theorem \ref{mainthm} - which 
shows that the  the only operators on $\ell_1$ that are not commutators are the ones of the form $\lambda I + K$, where $K$  is a compact operator 
and $\lambda\neq 0$. 
In the last section we give a characterization of the commutators on $\ell_{p_1}\oplus\ell_{p_2}\oplus\cdots\oplus\ell_{p_n}$, where 
$1\leq p_n < p_{n-1}< \ldots <p_1 < \infty$ and we also show that the strictly singular operators on $\ell_{\infty}$ are commutators. 

The author thanks W.B. Johnson for useful conversations and suggestions on the subject of this paper.
\section{Notation and basic results} 

We will follow the ideas in Apostol's paper \cite{Apostol_lp}, which in turn extend those of Brown \& Pearcy \cite{BrownPearcy} and earlier work 
referenced in \cite{BrownPearcy}, to develop a base for investigating 
the commutators on $\ell_1$ and $\LP\,(1\leq p <\infty$).\\
Consider a fixed decomposition $\displaystyle \decomp =\{X_i\} $ of $\X$, $1\leq p<\infty$ or $p=0$, meaning 
$\displaystyle \X\simeq \left ( \bigoplus_{i=0}^{\infty} X_i\right)_{p}$, 
where $X_i$ are complemented subspaces of $\X$ which are also isomorphic to $\X$. Let $\psi_i : X_i\to \X$ be an isomorphism
and let us also assume that $\|\psi_i^{-1}\| = 1$ for $i\in\N$ and $\displaystyle \lambda = \sup_{i\in\N}\|\psi_i\| < \infty$.
Denote by $P_i$ the projection from
$\X$ onto $X_i$. Let us also have a collection of uniformly bounded isomorphisms
$\{\varphi_i\}_1^{\infty}$ as shown below 
$$
X_0 \stackrel{\varphi_0}{\longrightarrow} X_1  \stackrel{\varphi_1}{\longrightarrow} 
X_2 \stackrel{\varphi_2}{\longrightarrow} X_3 \stackrel{\varphi_3}{\longrightarrow} \ldots
$$
which make the following diagram commute
$$
\begin{diagram}[height=2em,width=4.5em,abut] 
X_i& \rTo^{\varphi_i} & X_{i+1}\\
& \rdTo_{\psi_i} &\dTo > {\psi_{i+1}} \\
&& \X
\end{diagram}
$$
for every $i\in\N$. Clearly $\varphi_i = \psi_{i+1}^{-1}\circ\psi_i$ and $\|\varphi_i\|\leq\|\psi_i\|\leq\lambda, \,\,\|\varphi^{-1}_i\|\leq\lambda$.
Assume that for $p\geq 1$  we also have $\displaystyle \|\sum_{j=1}^nP_{i_j} x\|^p = \sum_{j=1}^n\|P_{i_j} x\|^p $ for every $n>0$ and $i_1,i_2,\ldots i_n > 0$
with $i_l\neq i_k$ for $l\neq k$. 
Note that using the last condition in the case $p\geq 1$ we have 
$\|P_i\|\leq \|I-P_0\| \leq \|P_0\| + 1 = C_1$. The last inequality is clearly true for $p=0$. 
\noindent
For $x=(x_i)\in \X$ , $x_i\in X_i$ define the following two operators :
$$
\RS_{\decomp} (x) = \sum_{i=0}^{\infty} \varphi_i (x_i) \qquad ,\qquad \LS_{\decomp} (x) = \sum_{i=0}^{\infty} \varphi^{-1}_i (x_{i+1}).
$$
The operators $L_{\decomp}$ and $R_{\decomp}$ are, respectively, the left and the right shift associated with the decomposition $\decomp$. As one may observe, the operators
$\LS_{\decomp}$ and $\RS_{\decomp}$ move the components of $x$ one position to the left/right, respectively, via the isomorphisms $\varphi_i$. Throughout the paper
we will simply use the letters $L$ and $R$ for the shifts when the decomposition $\decomp$ associated with the shifts is clear 
from the context. Our first proposition shows some basic properties of the left and the right shift as well as the fact that all the powers of $L$ and $R$
are uniformly bounded, which will play an important role in the sequel.
\begin{prop}\label{basicprop}
Let $\decomp$ be a decomposition of $\X$. Then we have
\begin{equation}\label{eq:LRbound}
\|L^n\|\leq 2\lambda C_1  \,\,,\,\, \|R^n\|\leq 2\lambda C_1 \,\,\,\textrm{for every}\,\, n = 1,2,\ldots
\end{equation}
\begin{equation}\label{eq:LRidentities}
\LS\RS = I  \,\,,\,\,\RS\LS = I-P_0\,\,,\,\, \RS P_i = P_{i+1}\RS \,\,\, , \,\,\, P_i\LS = \LS P_{i+1} \,\,\, \textrm{for}\,\,\, i\geq 0 .
\end{equation}
\begin{equation}\label{eq:Lgoto0}
\lim_{n\to\infty}\|L^n(x)\|_p = 0 \,\,\textrm{for all}\,\, 1\leq p<\infty \,\,\textrm{and}\,\, p = 0.
\end{equation}
\end{prop}
\begin{proof}
The relations $\LS\RS = I\,,\,\RS\LS = I-P_0$ are clear from the definition of the left and the right shift.
For $x=(x_i)\in \X$ , $x_i\in X_i$ we have 

\begin{eqnarray*}
P_{i+1}\RS (x) &=& P_{i+1}\left ( \sum_{i=0}^{\infty} \varphi_i (x_i)\right ) = \varphi_i(x_i) = \RS P_i (x) \\
P_i\LS (x) &=& P_i\left ( \sum_{i=0}^{\infty} \varphi^{-1}_i (x_{i+1}) \right ) = \varphi^{-1}_i (x_{i+1}) = \LS P_{i+1} (x) .
\end{eqnarray*}
Now using the fact that 

\begin{eqnarray*}
\varphi_{k+n}\circ\varphi_{k+n-1}\circ\cdots\circ\varphi_{k} &=&  \psi_{k+n+1}^{-1}\circ\psi_{k}\\
\varphi_{k}^{-1}\circ\varphi_{k+1}^{-1}\circ\cdots\circ\varphi_{k+n}^{-1} &=&   \psi_{k}^{-1}\circ\psi_{k+n+1} 
\end{eqnarray*}
we deduce that 

\begin{eqnarray*}
R^{n+1}(x) &=& \sum_{i=0}^{\infty} \varphi_{i+n}\circ\varphi_{i+n-1}\circ\cdots\circ\varphi_{i}(x_i) = 
\sum_{i=0}^{\infty}  \psi_{i+n+1}^{-1}\circ\psi_{i} (x_i)\\
L^{n+1}(x) &=& \sum_{i=0}^{\infty} \varphi_{i}^{-1}\circ\varphi_{i+1}^{-1}\circ\cdots\circ\varphi_{i+n}^{-1}(x_{i+n+1}) = 
\sum_{i=0}^{\infty} \psi_{i}^{-1}\circ\psi_{i+n+1}(x_{i+n+1}) 
\end{eqnarray*}
for $n=0,1,\ldots$. Finally we estimate the norms of $\|R^n\|$ and $\|L^n\|$ for $p\geq 1$:

\begin{eqnarray*}
\|R^{n+1}(x)\|^p_p &=& \|\sum_{i=0}^{\infty}  \psi_{i+n+1}^{-1}\circ\psi_{i}(x_i)\|^p\leq \sum_{i=0}^{\infty}\|\psi_{i+n+1}^{-1}\circ\psi_{i}\|^p\|x_i\|^p\\
&\leq& \sum_{i=0}^{\infty}  \lambda^p\|x_i\|^p = \lambda^p\|P_0x\|^p + \lambda^p\|\sum_{i=1}^{\infty} P_ix\|^p \\ \label{eq:Rnormbound}
&=& \lambda^p\|P_0x\|^p + \lambda^p\|(I-P_0)x\|^p\leq 2^p\lambda^p(\|P_0\|+1)^p\|x\|^p\\
\|L^{n+1}(x)\|^p_p &=& \|\psi_{0}^{-1}\circ\psi_{n+1}(x_{n+1}) + \sum_{i=1}^{\infty} \psi_{i}^{-1}\circ\psi_{i+n+1}(x_{i+n+1})\|^p \\
&\leq& 2^p(\|\psi_{0}^{-1}\circ\psi_{n+1}(x_{n+1})\|^p + \| \sum_{i=1}^{\infty} \psi_{i}^{-1}\circ\psi_{i+n+1}(x_{i+n+1})\|^p) \\ 
&\leq& 2^p(\lambda^p\|P_{n+1}x\|^p + \sum_{i=1}^{\infty} \lambda^p\|P_{i+n+1}x\|^p) \leq
2^p\lambda^p \sum_{i=1}^{\infty} \|P_{i}x\|^p \\
&=& \|2\lambda\sum_{i=1}^{\infty} P_{i}x\|^p = 2^p\lambda^p\|(I-P_0)x\|^p \leq 2^p\lambda^p(\|P_0\|+1)^p\|x\|^p.
\end{eqnarray*} 
Note that $\displaystyle \|\LS^m (x)\|^p_p \leq 2^p\lambda^p\|\sum_{i=0}^{\infty} P_{i+m}x\|^p\to 0$ shows (\ref{eq:Lgoto0}) for $p\geq 1$.
In the case $p=0$ the computations are somewhat simpler and are shown below:

\begin{eqnarray*}
\|R^{n+1}(x)\|_{\infty} &=& \|\sum_{i=0}^{\infty}  \psi_{i+n+1}^{-1}\circ\psi_{i}(x_i)\|_{\infty} = \max_{0\leq i<\infty}\|\psi_{i+n+1}^{-1}\circ\psi_{i}(x_i)\|\\
&\leq& \max_{0\leq i<\infty}\|\psi_{i+n+1}^{-1}\circ\psi_{i}\|\|x_i\| \leq \lambda \max_{0\leq i<\infty} \|x_i\| = \lambda\|x\|_{\infty}\\
\|L^{n+1}(x)\|_{\infty} &=& \|\sum_{i=0}^{\infty} \psi_{i}^{-1}\circ\psi_{i+n+1}(x_{i+n+1})\|_{\infty} \\
&=& \max_{0\leq i<\infty}\|\psi_{i}^{-1}\circ\psi_{i+n+1}(x_{i+n+1})\| \\
&\leq& \max_{0\leq i<\infty}\|\psi_{i}^{-1}\circ\psi_{i+n+1}\|\|x_{i+n+1}\|\leq \lambda \max_{0\leq i<\infty}\|x_{i+n+1}\|\\
&\leq&\lambda\|x\|_{\infty}.
\end{eqnarray*} 
In this case $\displaystyle \|\LS^m (x)\|_0\leq \lambda \max_{0\leq i<\infty}\|x_{i+n+1}\|\to 0$ proves (\ref{eq:Lgoto0}) for $p=0$. 
\end{proof}
Denote by $D_S$ the inner derivation determined by $S$ in $\opX$ i.e.
$$
D_ST = ST - TS .
$$
In the notation introduced above, an operator $T\in \opX$ is a commutator if and only if there exists $S\in \opX$ such that $T\in D_S\opX$.
For a given decomposition $\decomp$ of $\X$ denote by $\AD$ the set
\begin{equation}\label{eq:ADDef}
\AD = \{T\in \opX : \sum_{n=0}^{\infty} \RS^nT\LS^n \,\,\,\textrm{is strongly convergent}\}. 
\end{equation}
For $T\in\AD$ we  define
$$
\TD = \sum_{n=0}^{\infty} \RS^nT\LS^n .
$$
Our next lemma shows that each operator $T\in\AD$ is a commutator and also gives an explicit expression for $T$ as the commutator of two operators.
\begin{lem}\label{EasyLemma}
Let $T\in\AD$ for some decomposition $\decomp = \{X_i\}$ of $\X$. Then we have
\begin{equation}\label{eq:TinADRepresentation}
T = D_{\LS}(\RS\TD) = -D_{\RS}(\TD\LS) .
\end{equation}
\end{lem}
\begin{proof}
We will show one of the equalities via direct computation. The proof of the other is similar.

\begin{eqnarray*}
D_{\LS}(\RS\TD) &=& \LS\RS\TD - \RS\TD\LS = \TD - \RS(\sum_{n=0}^{\infty} \RS^nT\LS^n)\LS \\
&=& \TD - \sum_{n=1}^{\infty} \RS^nT\LS^n = T .
\end{eqnarray*}
In the computation above we used the convention $\LS^0 = \RS^0 = I$.
\end{proof}
\begin{lem}\label{EasyLemma2}
For a decomposition $\decomp = \{X_i\}$ of $\X$ we have the following relations
$$
\AD = D_{\RS}(\opX\RS\LS) = D_{\LS}(\RS\LS\opX) .
$$
\end{lem}
\begin{proof}
We will show the first of the relations. The proof of the second, as one may expect, is similar.\\
If $T\in\AD$, then $\TD\LS = \TD\LS\RS\LS = (\TD\LS)\RS\LS \in \opX\RS\LS$. Then using 
$T = -D_{\RS}(\TD\LS)$ from (\ref{eq:TinADRepresentation}) we have $T\in D_{\RS}(\opX\RS\LS)$.\\
Now, to prove the other direction, assume that $T\in \opX\RS\LS$. Then $T = S\RS\LS$ for some operator $S$ (hence $T\RS = S\RS$).
Then 

\begin{eqnarray*}
\sum_{n=0}^{m} \RS^n(D_{\RS}T)\LS^n &=&  \sum_{n=0}^{m} \RS^n(\RS T - T\RS)\LS^n = \sum_{n=0}^{m} \RS^{n+1}T\LS^n - \sum_{n=0}^{m} \RS^{n}T\RS\LS^n  \\
&=& \sum_{n=0}^{m} \RS^{n+1}S\RS\LS\LS^n -  \sum_{n=0}^{m} \RS^{n}S\RS\LS\RS\LS^n \\ 
&=& \sum_{n=0}^{m} \RS^{n+1}S\RS\LS^{n+1} -  \sum_{n=0}^{m} \RS^{n}S\RS\LS^n  \\
&=& \RS^{m+1}S\RS\LS^{m+1} -S\RS = \RS^{m+1}T\LS^{m} - T\RS .
\end{eqnarray*}
Since $\displaystyle \lim_{m\to\infty}\|\LS^m (x)\| = 0$ for any $x\in\X$ from (\ref{eq:Lgoto0}) and $\|R^m\|<2\lambda C_1$ for every $m>0$,  
we have \\
$\displaystyle \lim_{m\to\infty}\sum_{n=0}^{m} \RS^n(D_{\RS}T)\LS^n = -T\RS$. From the last equation we conclude that 
$D_{\RS}T\in \AD$ and $(D_{\RS}T)_{\decomp} = -T\RS$.  Moreover, from $T\RS = S\RS$ we have $(D_{\RS}T)_{\decomp} = -S\RS$ and multiplying both sides 
by $\LS$ we obtain $(D_{\RS}T)_{\decomp}\LS = -T$.
\end{proof}
We proved that for a given decomposition $\decomp$ all operators $T\in\AD$ are commutators, but in general the condition in $(\ref{eq:ADDef})$ 
is hard to check for a given operator $T$. We want to have a
condition on $T$ which is easy to check and which ensures the containment $T\in\AD$. To be more precise, given an operator $T$,
we want to have a condition on $T$ which will allow us to build a decomposition $\decomp$ for which $T\in\AD$. Our next lemma gives us such a condition (as 
will become clear later) and it will be our main tool for constructing decompositions in the sequel.

\begin{lem}\label{MainLemma}
Let $T\in\opX$ and  $\decomp = \{X_i\}$ be a decomposition of $\X$. Fix $\varepsilon >0$ and denote
$\displaystyle \widetilde{P}_n = \sum_{i=0}^{n} P_i$, where $P_i$ is the projection onto $X_i$. Let us also assume that 
\begin{equation}\label{eq:maincondition}
\lim_{n\to\infty}\|(I - \widetilde{P}_n)T\| = \lim_{n\to\infty}\|T(I - \widetilde{P}_n)\| = 0 .
\end{equation}
Then there exists an increasing sequence of integers $\displaystyle \{m_j\}_{i=0}^{\infty}$ such that
$$
\sum_{j=0}^{\infty}\|(I-\widetilde{P}_{m_j})T\| + \sum_{j=0}^{\infty}\|T(I-\widetilde{P}_{m_j})\| 
+ \sum_{i,j=0}^{\infty}\|(I-\widetilde{P}_{m_i})T(I-\widetilde{P}_{m_j})\| < \varepsilon .
$$
\end{lem}
\begin{proof}
Note first that $\|I-\widetilde{P}_{i}\| \leq \|P_0\| + 1 = C_1$ for every $i\in\N$.
Let $\displaystyle \{n_j\}_{j=0}^{\infty}$ be an increasing sequence of integers such that
$$
\sum_{j=0}^{\infty}\|T(I-\widetilde{P}_{n_j})\| < \frac{\varepsilon}{3C_1}\,\, , 
\,\, \sum_{j=0}^{\infty}\|(I-\widetilde{P}_{n_j})T\| < \frac{\varepsilon}{3C_1}.
$$
Now we can use the inequality
$$
\sum_{j = 0}^{\infty}\|(I-\widetilde{P}_{i})T(I-\widetilde{P}_{n_j})\| \leq \sum_{j = 0}^{m}\|(I-\widetilde{P}_{i})T(I-\widetilde{P}_{n_j})\|  + 
C_1\sum_{j = m+1}^{\infty}\|T(I-\widetilde{P}_{n_j})\| 
$$
to deduce that 
$$
\lim_{i\to\infty}\sum_{j=0}^{\infty}\|(I-\widetilde{P}_{i})T(I-\widetilde{P}_{n_j})\| = 0 .
$$
Using the last equation we can find an increasing sequence of integers \\
$\displaystyle \{m_j\}_{j=0}^{\infty},\, m_j \geq n_j$ such that 
$$
\sum_{i=0}^{\infty}\sum_{j=0}^{\infty}\|(I-\widetilde{P}_{m_i})T(I-\widetilde{P}_{n_j})\| < \frac{\varepsilon}{3C_1} .
$$
Now it is easy to deduce that the sequence $\displaystyle \{m_j\}_{j=0}^{\infty}$ satisfies the condition of the lemma. In fact

\begin{eqnarray*}
\|T(I-\widetilde{P}_{m_j})\| &=& \|T(I-\widetilde{P}_{n_j})(I-\widetilde{P}_{m_j})\|\leq C_1\|T(I-\widetilde{P}_{n_j})\| \\
\|(I-\widetilde{P}_{m_j})T\| &=& \|(I-\widetilde{P}_{m_j})(I-\widetilde{P}_{n_j})T\|\leq C_1 \|(I-\widetilde{P}_{n_j})T\| \\
\|(I-\widetilde{P}_{m_i})T(I-\widetilde{P}_{m_j})\| &=& \|(I-\widetilde{P}_{m_i})T(I-\widetilde{P}_{n_j})(I-\widetilde{P}_{m_j})\|\\
&\leq& C_1\|(I-\widetilde{P}_{m_i})T(I-\widetilde{P}_{n_j})\| .
\end{eqnarray*}
This finishes the proof.
\end{proof}

\begin{lem}\label{PTPLemma}
Let $\decomp = \{X_i\}$ be a decomposition of $\X$. Then for any $T\in\opX$ we have
$$
P_iTP_j\in\AD \, , \, \|(P_iTP_j)_\decomp\|\leq C\|P_iTP_j\|
$$
where $C$ depends on $\decomp$ only.
\end{lem}
\begin{proof}
Let us consider the case $p\geq 1$ first.
Note first that $\|L^n\|\leq 2\lambda C_1$ and $\|R^n\|\leq 2\lambda C_1$. For any $x\in\X$  we have (see Prop. (\ref{basicprop}))

\begin{eqnarray*}
\|\sum_{n=m}^{m+r} \RS^{n}P_iTP_j\LS^nx\|^p &=&  \|\sum_{n=m}^{m+r} \RS^{n}P_iTP_j\LS^nP_{j+n}x\|^p  \\
&\leq& 4\lambda^2C_1^2\|P_iTP_j\|^p\sum_{n=m}^{m+r}\|P_{j+n}x\|^p \\
&\leq&  4\lambda^2C_1^2\|P_iTP_j\|^p\sum_{n=m}^{\infty}\|P_{j+n}x\|^p  \\
&\leq& 4\lambda^2C_1^3\|P_iTP_j\|^p\|x\|^p .
\end{eqnarray*}
Since $\displaystyle \sum_{n=m}^{\infty}\|P_{j+n}x\|_p^p\to 0$ as $m\to\infty$ we have that 
$\displaystyle \sum_{n=0}^{\infty} \RS^{n}P_iTP_j\LS^n$ is strongly convergent and $P_iTP_j\in\AD$. The inequality in the Theorem in this case
 follows from the inequality above with $C = 4\lambda^2C_1^3$.\\
For $p=0$ a similar calculation shows 

\begin{eqnarray*}
\|\sum_{n=m}^{m+r} \RS^{n}P_iTP_j\LS^nx\|_{\infty} &=&  \|\sum_{n=m}^{m+r} \RS^{n}P_iTP_j\LS^nP_{j+n}x\|_{\infty}  \\
&=& \max_{m\leq n\leq m+r} \|\RS^{n}P_iTP_j\LS^nP_{j+n}x\|\\
&\leq&  4\lambda^2C_1^2\|P_iTP_j\|\max_{m\leq n\leq m+r}\|P_{j+n}x\|.  
\end{eqnarray*}
Since $\displaystyle \max_{m\leq n\leq m+r}\|P_{j+n}x\|\to 0$ as $m\to\infty$  we apply the same argument as in the case $p\geq 1$
 to obtain the conclusion of the theorem. 
\end{proof}
\begin{cor} \label{largekernel}
Let $T\in\opX$ and $\decomp = \{X_i\}$ be a decomposition of $\X$. Then we have

\begin{eqnarray*}
TP_0 &=& D_{\RS}(\LS TP_0 - (P_0TP_0)_{\decomp}\LS)\\
P_0T &=& D_{\LS}(-P_0T\RS + \RS(P_0TP_0)_{\decomp}) .
\end{eqnarray*}
\end{cor}
\begin{proof}
We will prove the first equation. Note that from Lemma \ref{EasyLemma} we have $-D_{\RS}((P_0TP_0)_{\decomp}\LS) = P_0TP_0$ and
$D_{\LS}(\RS(P_0TP_0)_{\decomp}) =  P_0TP_0$.
Now 

\begin{eqnarray*}
D_{\RS}(\LS TP_0 - (P_0TP_0)_{\decomp}\LS) &=& \RS\LS TP_0 - \LS TP_0\RS + P_0TP_0  \\
&=&(I-P_0) TP_0  + P_0TP_0 = TP_0\\
D_{\LS}(-P_0T\RS + \RS(P_0TP_0) &=&-\LS P_0T\RS + P_0T\RS\LS + P_0TP_0  \\
&=&P_0T(I - P_0) + P_0TP_0 =P_0T .
\end{eqnarray*}
Above we used the equality $P_0\RS = 0 = \LS P_0$, which is clear from the definitions of $R$ and $L$.
\end{proof}
The following theorem shows the importance of the decompositions in determining whether an operator is a commutator.
\begin{thm}\label{thm:decomp}
Under the hypotheses of Lemma \ref{MainLemma}, there is a decomposition $\decomp$ of $\X$ for which
$$
T\in\AD ,\,\,\, \|T_{\decomp}\| \leq C\|T\| + \varepsilon
$$
where $C$ depends on $\decomp$ only. In particular, using Lemma \ref{EasyLemma} we conclude that $T$ is a commutator.
\end{thm}
\begin{proof}
Using the sequence $\{m_j\}$ from Lemma \ref{MainLemma}, define a new decomposition where $\displaystyle \widetilde{X}_0 = \bigoplus_{k=0}^{m_0} X_k$, 
$\displaystyle \widetilde{X}_i = \bigoplus_{k=m_{i-1}+1}^{m_i} X_k$. Note that the new decomposition also satisfies the condition about the projections.
 For simplicity of notation we will denote the new decomposition by $\{X_i\}$ and the 
projections onto $X_i$ by $P_i$. In the new notation the conclusion from Lemma \ref{MainLemma} can be written as 
$$
\sum_{j=0}^{\infty}\|(I-\widetilde{P}_{j})T\| + \sum_{j=0}^{\infty}\|T(I-\widetilde{P}_{j})\| 
+ \sum_{i,j=0}^{\infty}\|(I-\widetilde{P}_{i})T(I-\widetilde{P}_{j})\| < \varepsilon .
$$
Now using $P_i(I-\widetilde{P}_{i-1}) = (I-\widetilde{P}_{i-1})P_i = P_i$ we have 

\begin{eqnarray*}
\sum_{i,j = 0}^{\infty} \|P_iTP_j\| &\leq& \|P_0TP_0\| + C_1\sum_{i = 1}^{\infty} \|P_iT\| + C_1\sum_{j = 1}^{\infty} \|TP_j\| +
\sum_{i,j = 1}^{\infty} \|P_iTP_j\| \\
&\leq& \|P_0TP_0\| + C_1\sum_{i = 1}^{\infty} \|P_i(I-\widetilde{P}_{i-1})T\| \\
&+& C_1\sum_{j = 1}^{\infty} \|T(I-\widetilde{P}_{j-1})P_j\| \\
&+& \sum_{i,j = 1}^{\infty} \|P_i(I-\widetilde{P}_{i-1})T(I-\widetilde{P}_{j-1})P_j\|  \\
&\leq& \|P_0TP_0\| + C_1^2\sum_{i = 1}^{\infty} \|(I-\widetilde{P}_{i-1})T\| + C_1^2\sum_{j = 1}^{\infty} \|T(I-\widetilde{P}_{j-1})\|  \\
&+& C_1^2\sum_{i,j = 1}^{\infty} \|(I-\widetilde{P}_{i-1})T(I-\widetilde{P}_{j-1})\|  \\
&\leq& \|P_0TP_0\| + C_1^2\varepsilon .
\end{eqnarray*}

Since the series $\displaystyle \sum_{i=0}^{\infty}P_i$ is strongly convergent to $I$, we have  $\displaystyle T = \sum_{i,j = 0}^{\infty} P_iTP_j$ in the
norm topology of $\opX$. Using Lemma \ref{PTPLemma}, the operator 
$$
S = \sum_{i,j = 0}^{\infty} (P_iTP_j)_{\decomp}
$$ 
is well defined and using Lemma \ref{EasyLemma} we have that $T = D_{\RS}(-S\LS)\in\AD$. Now $D_{\RS}(\TD\LS - S\LS) = 0$ and by the proof of
Lemma \ref{EasyLemma2} we have
$$
0 = -(D_{\RS}(\TD\LS - S\LS))_{\decomp}\LS = (\TD - S)\LS .
$$
From the equation above we conclude that $\TD = S$ and $\|\TD\| \leq C\|T\| + \varepsilon$.
\end{proof}

\section{Compactness and commutators on $\ell_p$ and $L_p\,(1\leq p <\infty)$}
In order to prove the conjecture about the structure of the commutators on a given space we have to show that all the elements in the unique maximal 
ideal are commutators. We prove a lemma that takes care of this in the case $\X = \ell_1$ and also shows that the ideal of compact operators consists of 
commutators only, provided the space $\X$ has some additional structure. Before that we will show a lemma about the operators $T$ on $\X$ 
which do not preserve a copy of $\X$ in the cases of $\X = \ell_1$ and $\X = L_1$, which we will use and is interesting on its own. 
\begin{lem}\label{SmallNormLemma}
Let $\X = L_1$ or $\X = \ell_1$ and suppose that $T\in\opX$ does not preserve a copy of $\X$. Then, for every $\delta >0$ and for every $\tilde{X}\subset\X$, 
$\tilde{X}\equiv\X$, there exists $Y\subset\tilde{X}$, such that  $Y$ is $(1+\delta)$ isomorphic to $\X$, $(1+\delta)$ complemented in $\X$, and 
$\|T_{|Y}\|<\delta$.
\end{lem}
\begin{proof}
Consider the case $\X = L_1$ first. By assumption $T$ does not preserve a copy of $L_1$ which implies that $T$ is not an $E$-operator (actually this can 
be taken as an equivalent definition for an operator not to be an $E$-operator \cite[Theorem 4.1]{Enflo_Starbird}) 
and hence it is not sign-preserving either (\cite[Theorem 1.5]{Rosenthal_L1}). Now \cite[Lemma 3.1]{Rosenthal_L1} gives us a subspace $Z\subset \tilde{X}$
 such  that $Z\simeq \tilde{X}$ and $\|T_{|Z}\|<\delta$. Using \cite[Theorem 1.1]{Rosenthal_L1} we find $Y\subset Z$,  which is 
$(1+\delta)$ isomorphic to $\tilde{X}\equiv L_1$, $(1+\delta)$ complemented in $\tilde{X}$ and $Y$ clearly satisfies $\|T_{|Y}\|<\delta$.
If $Q$ is the norm one projection onto $\tilde{X}$, and $R: \tilde{X}\to Y$ is a projection of norm less than $1+\delta$, then 
$P: = RQ$ is a projection from $L_1$ onto $Y$ and $\|P\|<1+\delta$.\\
For the case $\X = \ell_1$ we use the fact that if $\tilde{X}$ is isometric to $\ell_1$, 
then $\tilde{X} = \overline{\textrm{span}}\{\psi_i\,:\, i = 0,1,\ldots\}$ for some  vectors $\{\psi_i\}_{i=1}^{\infty}$ of norm one, such that 
$$
\psi_j = \sum_{i\in\sigma_j}\lambda_ie_i, \qquad\textrm{with}\,\,\sigma_j\cap\sigma_k = \emptyset\,\,\textrm{for}\,\, j\neq k
$$
where $\{e_i\}_{i=1}^{\infty}$ is the standard unit vector basis of $\ell_1$. This follows trivially from the observation that
$Ue_i$ and $Ue_j$ must have disjoint supports if $U:\tilde{X}\to\ell_1$ is an into isometry (cf. \cite[Proposition 2.f.14]{LT}).
 Note also that since every infinite dimensional subspace of 
$\ell_1$ contains an isomorphic copy of $\ell_1$ (\cite[Proposition 2.a.2]{LT}), then the operator $T$ is automatically strictly singular and hence compact (\cite{GM}).
Then, $\{T\psi_i\}_{i=0}^{\infty}$ is relatively compact in
$\ell_1$ and hence there exist $y\in\ell_1$ and a subsequence $\{\psi_{i_j}\}$ such that $T\psi_{i_j}\to y$. WLOG we may assume that 
$T\psi_i\to y$. Finally, define $\displaystyle \varphi_i = \frac{\psi_{2i} - \psi_{2i+1}}{2}$ for $i=0,1,\ldots$. Clearly
 $\{\varphi_i\}_{i=0}^{\infty}$ is a normalized block basis of $\tilde{X}$ such that $\|T\varphi_i\|_1\to 0$.
Assume WLOG that $\|T\varphi_i\|_1<\varepsilon$ (this can be easily achieved by passing to a subsequence).
Then for $Y = \overline{\textrm{span}}\{\varphi_i\,:\, i = 0,1,\ldots\}$ we have $\|T_{|Y}\|<\varepsilon$. Note also that $Y\subset \tilde{X}$
 is 1-complemented in $\tilde{X}$ as it is the closed span of a normalized block basis and clearly is isometric to $\tilde{X}$ (\cite[Lemma 1]{Pelcz_pojections}). 
Finally, let $R: \tilde{X}\to Y$ be the norm one projection onto $Y$ and $Q:\ell_1\to\tilde{X}$ be the norm one projection onto $\tilde{X}$.
 Then clearly $P: = RQ$ is a norm one projection onto $Y$.
\end{proof}

\begin{lem}\label{CompactLemma}
Let $\X$ be a Banach space for which $\displaystyle  \X\simeq \left ( \bigoplus_{i=0}^{\infty} \X\right)_{p}$ for some $1\leq p<\infty$ or $p=0$.
In the case $p=1$ we  will assume  that $\X = L_1$ or $\X = \ell_1$.
 Let $T\in\opX$ be a compact operator and $\varepsilon > 0$. Then there exists a decomposition $\decomp$ of $\X$  such that 
$T\in\AD$ and $\|\TD\| \leq C\|T\| + \varepsilon$ for some constant $C$ depending on $\decomp$ only. Consequently $T$ is a commutator and  $T = -D_{\RS}(\TD\LS)$.
\end{lem}
\begin{proof}
The result is known in the case of $\X = \LP$ and $\X = \ell_p$ for $1 < p < \infty$ (cf. \cite{Schneeberger} and \cite{Apostol_lp}), and for $\X = c_0$ 
and $\X = C(K)$ (\cite{Apostol_c0}). 
The proof presented here in these cases follows Apostol's ideas from \cite{Apostol_lp} and our generalized context gives a shorter proof in the case of 
$\LP$ for $1 < p < \infty$. Partial results were known in the case $\X = \ell_1$ (\cite[Theorem 2.6]{Apostol_lp}).

{\bf{Case I.}} $p>1$ or $p=0$. In this case we proceed as in Theorem 2.4 in \cite{Apostol_lp}. 
Consider an arbitrary decomposition $\decomp$ of $\X$ and denote $\displaystyle \widetilde{P}_n = \sum_{i=0}^{n} P_i$. Now we have 
$$
\lim_{n\to\infty}\|(I - \widetilde{P}_n)T\| = \lim_{n\to\infty}\|T(I - \widetilde{P}_n)\|= 0 .
$$
In fact if we choose $\varphi_i, \psi_i\in\X$ such that 
$$
\|(I - \widetilde{P}_n)T\varphi_n\|_{\X} > \|(I - \widetilde{P}_n)T\| - \frac{1}{n+1}, \,\, \|\varphi_n\| = 1
$$
$$
\|T(I - \widetilde{P}_n)\psi_n\|_{\X} > \|T(I - \widetilde{P}_n)\| - \frac{1}{n+1}, \,\, \|\psi_n\| = 1,\,\, (I - \widetilde{P}_n)\psi_n = \psi_n .
$$
Since the set $\{T\varphi_i\}_{i=0}^{\infty}$ is relatively compact in $\X$ and the sequence $\{(I - \widetilde{P}_i)\}_{i=0}^{\infty}$
 converges strongly to $0$ we have $\displaystyle \lim_{n\to\infty}\|(I - \widetilde{P}_n)T\| = 0$. 
On the other hand, the sequence $\{\psi_i\}_{i=0}^{\infty}$ is weakly convergent  to $0$. Now using the fact that $T$ is compact, it follows that the 
sequence $\{T\psi_i\}_{i=0}^{\infty}$ converges to $0$ in norm and hence $\displaystyle \lim_{n\to\infty}\|T(I - \widetilde{P}_n)\|= 0$.
Now Theorem \ref{thm:decomp} gives the result.

{\bf{Case II.}} $p=1$. 
Fix $\varepsilon >0$ and let $\displaystyle \decomp =\{X_i\} $ be the fixed decomposition of $\X$ defined by 
$\displaystyle X_i = L_1[\frac{1}{2^{i+1}}, \frac{1}{2^i})$ in the  case of $\X = L_1$ and by $\displaystyle X_i = P_{N_i}\ell_1$ 
(where $\displaystyle \N = \cup_{i=0}^{\infty} N_i$ such that $\textrm{card}\, N_i = \textrm{card}\,\N$ for all $i\in\N$ and $N_j\cap N_j=\emptyset$
for $i\neq j$) in the case of $\X = \ell_1$. Using Lemma \ref{SmallNormLemma} for each $X_i$ with $\displaystyle \delta= \frac {\varepsilon}{2^i}$ will
give us $1+\varepsilon$ complemented subspaces $\{Y_i\}$ of $\X$ which are isomorphic to $\X$ and $\displaystyle \|T_{|Y_i}\|<\frac {\varepsilon}{2^i}$.
Set $\displaystyle Y_0 = (I - \sum_{i=1}^{\infty} P_i)\X$. Note that $\decomp = \{Y_i\}$ is a decomposition
for $\X$ since all the spaces are complemented and isomorphic to $\X$. This is clear for $Y_i$ for $i = 1, 2, \ldots$ and it also holds for 
$Y_0$, since $X_0\subset Y_0$ is complemented in $\X$, isomorphic to $\X$, and using \cite[Corollary 5.3]{Enflo_Starbird} in the case $\X = L_1$, and 
 \cite[Proposition 4]{Pelcz_pojections} in the case $\X = \ell_1$, it follows that $Y_0$ is  isomorphic to $\X$ as well.  
Now, if  $\displaystyle \widetilde{P}_n = \sum_{i=0}^{n} P_i$, then clearly we have 
$\displaystyle\lim_{n\to\infty}\|T(I - \widetilde{P}_n)\| = 0$. Since $T$ is  compact operator, then 
we have $\displaystyle\lim_{n\to\infty}\|(I - \widetilde{P}_n)T\| = 0$ as well (the argument provided in {\bf{Case I}}
above works in this case as well), so using Theorem \ref{thm:decomp} we conclude that 
$T$ is a commutator. 
\end{proof}
\begin{rem}
Using the previous lemma we immediately conclude that \cite[Theorem 4.3]{Schneeberger} holds for $p=1$. Namely, a multiplication operators $M_{\phi}$
on $L_1$ is a commutator if and only if the spectrum of $M_{\phi}$ contains more than one limit point or contains zero as the unique limit point.
\end{rem}
\begin{cor}Let $\X$ be a Banach space for which $\displaystyle  \X\simeq \left ( \bigoplus_{i=0}^{\infty} \X\right)_{p}$ for some $1\leq p<\infty$ or $p=0$. 
In  the case $p=1$ we will assume 
that $\X = L_1$ or $\X = \ell_1$.
Let $T\in\opX$ and suppose that $P$ is a projection on $\X$ such that $P\X\simeq \X\simeq (I-P)\X$ and that either $TP$ or $PT$ is a compact operator. 
Then $T$ is a commutator.
\end{cor}
\begin{proof}
First we treat the case when $TP$ is compact operator.
Let $\decomp = \{X_i\}_{i=0}^{\infty}$ be a decomposition for which $TP\in\AD$ and 
$\displaystyle \|(TP)_{\decomp}\|_{\X}\leq C\|TP\|_{\X} + \frac{\varepsilon}{2}$ for a fixed $\varepsilon >0$ (by Lemma\, \ref{CompactLemma}). 
We also want $\decomp$ to be such that $(I-P)\X =  X_0$ hence we may assume $(I-P) = P_0$, where $P_0$ is the projection onto $X_0$.
 This can obviously be done for $1 < p < \infty$
(since the decomposition used in the proof was arbitrary). In the case of $L_1$ we consider the operator $\tilde{T} = GTG^{-1}$
where $G : P\X\oplus (I-P)\X\to (I-P_0)\X\oplus X_0$ is an isomorphism such that $GP\X = (I-P_0)\X$, 
$G(I-P)\X = X_0$. In this case $\tilde{T}GPG^{-1}$ is compact and clearly we can choose the decomposition as in Lemma \ref{CompactLemma}
and apply the same argument.
Now WLOG we will assume that $\tilde{T} = T$. In the case of $\ell_1$ we can make a similarity as in the previous case and reduce to the case
where $TP_M$ is a compact operator for some $M\subset\N$. Define
$$
S = \LS T(I-P) - (P_0T(I-P)P_0)_{\decomp}\LS - (TP)_{\decomp}\LS .
$$
Use equation (\ref{eq:TinADRepresentation}) applied to $TP$ and $P_0T(I-P)P_0$ to get
\begin{eqnarray} \label{eq:p1}
-D_{\RS}((TP)_{\decomp}\LS) &=& TP \\ 
 -D_{\RS}((P_0T(I-P)P_0)_{\decomp}\LS) &=& P_0T(I-P)P_0 = P_0T(I-P).\label{eq:p2}
\end{eqnarray}
Now 
\begin{equation}\label{eq:p3}
D_{\RS}(\LS T(I-P)) = \RS\LS T(I-P) - \LS T(I-P)\RS = (I-P_0)T(I-P)
\end{equation}
since $(I-P)\RS = 0$. Combining (\ref{eq:p1}), (\ref{eq:p2}) and (\ref{eq:p3}) we conclude that $D_{\RS}S = T$.
If $PT$ is compact we consider $S = - (I-P)T\RS + \RS(P_0(I-P)TP_0)_{\decomp} + \RS(PT)_{\decomp}$ and a
similar calculation shows that $T = D_{\LS}(S)$.
\end{proof}
\section{Commutators on $\ell_1$} 

We already saw in the previous section that the compact operators on $\ell_1$ are 
commutators and in order to prove the conjecture in the case of $\X = \ell_1$ we have to show that all operators not of the form $\lambda I + K$, 
where $K$  is compact and $\lambda\neq 0$, are commutators. To do that we are going to show that if $T$ is not of the form $\lambda I + K$, then there
exist disjoint complemented subspaces $X$ and $Y$ of $\X$ which are isomorphic to $\X$ and such that $T_{|X}:X\to Y$ is an onto isomorphism.
As we will see, this last property of $T$ will be enough to show that $T$ is a commutator on any space $\X$  for which 
$\displaystyle \X\simeq \left ( \bigoplus_{i=0}^{\infty} \X\right)_{p}$ .

\begin{defn}\label{def:les}
The left essential spectrum  of $T\in\opX$ is the set (\cite{Apostol_les} Def 1.1)
\end{defn}
$$
\les = \{\lambda\in\Lambda : \inf_{\begin{array}{c} x\in Y \\ \|x\| = 1 \end{array}}\|(\lambda - T)x\| = 0, {\textrm{codim}}(Y)<\infty, Y - {\textrm{closed}}\}.
$$

For any $T\in\opX$, $\les$ is a closed non-void set (\cite[Theorem 1.4]{Apostol_les}).
The following lemma is an analog of Lemma 4.1 from \cite{Apostol_lp} and the proof follows the steps in the proof there. 
\begin{lem}\label{PTPLemma1}
Let $\X =\ell_1$ and let $T\in\opX$, $0\in\les$. If $T$ is not compact, then it is similar to an operator 
$T^{'}\in\opX$ for which there exists a projection $P_M$ such that  $M\subset\N$, $card\, M = card\,(\N - M)$, 
and $P_{\N - M}T^{'}P_M$ is not a compact operator. 
\end{lem}
\begin{proof}
For simplicity of notation we will denote $P_M$ simply by $P$.
By Lemma 3.2 in \cite{Apostol_lp} and using a similarity we can obtain a subset $M$ of $\N$ so that $\displaystyle \sum_{n\in\N - M}\|TP_{\{n\}}\| <  \infty$, where
$P_{\{n\}}$ is the projection onto the $n$-th coordinate. From this inequality we have that $T(I-P)$ is compact.
If $(I-P)TP$ is not compact we are done, thus we may suppose that $(I-P)TP$ is compact.
The equality
$$
T = T(I-P) + (I-P)TP + PTP
$$
gives us that $PTP$ is not compact. Using $\X \equiv P\X\oplus_1(I-P)\X$, let $\varphi : P\X\to (I-P)\X$ be an isometry
 and define the operators $V$ and $V^{'}$ in the following way
$$
V(x) = \varphi(Px), \qquad V^{'}x = \varphi^{-1}((I-P)x) .
$$
It is easy to see that $PV = V(I-P) = V^{'}P = (I-P)V^{'} = V^2 = (V^{'})^2 =  0$ and $VV^{'}+V^{'}V  = I$. Define
$$
\sqrt{2}S = P - (I-P) + V + V^{'} .
$$
Now a simple check gives us 

\begin{eqnarray*}
2S^2 &=& (P - (I-P) + V + V^{'})(P - (I-P) + V + V^{'}) \\
&=& P + PV^{'} + (I-P) - (I-P)V + VP + VV^{'} - V^{'}(I-P) + V^{'}V  \\
&=& 2 + PV^{'} - (I-P)V + VP - V^{'}(I-P)  \\
&=& 2 + PV^{'} - V + VP - V^{'} = 2,
\end{eqnarray*}
hence $S = S^{-1}$. Now consider the operator $2(I-P)S^{-1}TSP$. Again a simple calculation shows that 

\begin{eqnarray*}
2(I-P)S^{-1}TSP &=& (-(I-P) + (I-P)V)T(P + VP) \\
&=& (-(I-P) + V)T(P + VP)\\
&=& -(I-P)TP -(I-P)TVP + VTP + VTVP = \\
&=& -(I-P)TP - (I-P)T(I-P)VP \\
&+& VPTP + VT(I-P)VP = VPTP + K
\end{eqnarray*} 
where $K =  -(I-P)TP - (I-P)T(I-P)VP + VT(I-P)VP$ is a compact operator because $T(I-P)$ and $(I-P)TP$ are compact operators
(the first one by construction, the second by assumption).
Since $V_{|P\X}$ is an isometry, we conclude that $VPTP$ is not compact and hence $(I-P)S^{-1}TSP$ is not compact either.
Taking $T^{'}  = S^{-1}TS$ we finish the proof.
\end{proof}
\begin{prop}\label{auxprop}
Let $\X = \ell_1$ and $T\in\opX$ be such that there exists a projection $P$ such that $P\X\simeq \X$, $(I-P)\X\simeq \X$ and 
the operator $(I-P)TP$ is not compact. Then there exists a complemented subspace $Y\subset P\X$ such that
$Y\simeq\X$, $(I-P)TP_{|Y}$ is an isomorphism into and $(I-P)TP(Y)$ is complemented in $\X$.
\end{prop}
\begin{proof}
Clearly $(I-P)TP$ is not a strictly singular operator since in $\opX$ the ideal of compact operators and strictly singular operators coincide (\cite{GM}).
In particular, this implies that there exists an infinite dimensional subspace $Z\subset P\X$ such that $(I-P)TP$ is an isomorphism on $Z$. Consider the 
infinite dimensional subspace $(I-P)TPZ$.
Using \cite[Lemma 2]{Pelcz_pojections} (cf. also \cite[Proposition 2.a.2]{LT}) we conclude that there exists $U\subset (I-P)TPZ$ which is complemented in
$\X$ and isomorphic to $\X$. Clearly $(I-P)TP$ is an isomorphism on $((I-P)TP)^{-1}U$ and since $U$ is complemented in $\X$, we also have that 
$((I-P)TP)^{-1}U$ is complemented in $\X$ as well.
\end{proof}

\begin{thm}\label{mainauxresult}
Let $\X \simeq \left( \bigoplus_{i=0}^{\infty} \X\right)_{p}$, $1\leq p<\infty$ or $p=0$, and let $T\in\opX$ be an operator for which there exists a projection $P$ such that $P\X\simeq \X$, $(I-P)\X\simeq \X$ and 
there exists a complemented subspace $Y\subseteq P\X$ such that $Y\simeq\X$, $(I-P)TP_{|Y}$ is an isomorphism into and
$X = (I-P)TP(Y)$ is also complemented in $\X$. Then there exists a decomposition $\decomp$ such that $T$ is similar to a matrix operator
$$\left( \begin{array}{cc}
* & L  \\
* & * 
\end{array} \right)
$$
on $\X\oplus\X$, where $L$ is the left shift associated with $\decomp$.
\end{thm}
\begin{proof}
Let $P_X$ be a projection onto $X$. Note that  $X = P_X\X\subset (I-P)\X$ and hence $Y\subset P\X\subset (I-P_X)\X$. The previous observation shows that
the operator $P_XT(I-P_X)_{|Y}$ is an isomorphism from $Y$ onto $X$. Also, $X=P_X\X\simeq \X$ by the assumption of the theorem,
and $(I-P_X)\X$ contains $Y$ - a complemented copy of $\X$ and hence $(I-P_X)\X\simeq\X$ (using a result of Pelczynski 
\cite[Proposition 4]{Pelcz_pojections}). The observations we made imply that WLOG we may assume $P_X = I-P$. 
Consider two decompositions $\displaystyle \decomp_1 =\{X_i\} $ , $\displaystyle \decomp_2 =\{Y_i\} $ of $\X$ such that 
$X = Y_0 = X_1\oplus X_2\oplus\ldots$, \\
$X_0 = Y_1\oplus Y_2\oplus\ldots$  and $Y_1 = Y$.
Define a map $S$
$$
S\varphi = \LSo\varphi\oplus\LSt\varphi , \qquad \varphi\in\X
$$
 from $\X$ to $\X\oplus\X$. The map $S$ is invertible ($S^{-1}(a,b) = \RSo a + \RSt b$).
Just using the definition of $S$ and the formula for $S^{-1}$ it is easy to see that 

\begin{eqnarray*}
STS^{-1}(a,b) &=& ST(\RSo a + \RSt b) = S(T\RSo a + T\RSt b) \\
&=& (\LSo T\RSo a + \LSo T\RSt b)\oplus (\LSt T\RSo a + \LSt T\RSt b),
\end{eqnarray*}
hence
$$
STS^{-1} = \left( \begin{array}{cc}
* & \LSo T\RSt   \\
* & * 
\end{array} \right) .
$$
Let 
\begin{equation}\label{eq:A}
A=P_{Y_0}T\RSt = (I-P)T\RSt
\end{equation}
 and note that $A_{|P_{Y_0}\X} \equiv A_{|(I-P)\X} : (I-P)\X\to (I-P)\X$ is onto and invertible since $\RSt$ is an isomorphism on 
$P_{Y_0}\X$ and $\RSt(P_{Y_0}\X) = Y_1 = Y$.
Here we used the fact that $P_{Y_0}T$ is an isomorphism on $Y$ ($PY = Y$). Denote by $T_0$ the inverse of $A_{|P_{Y_0}\X} $
(note that $T_0$ is an automorphism on $(I-P)\X$) and consider 
$$
G = I + T_0(I-P) - T_0A .
$$
We will show that $G^{-1} =  A+P$. In fact, from the definitions of $A$ and $T_0$
 it is clear that 
\begin{equation}\label{eq:AProp}
AT_0(I-P) = T_0A(I-P) = I-P \,,\, PT_0 = PA = 0 \,,\, (I-P)A = A
\end{equation}
and since $A$ maps onto $(I-P)\X$ and ${AT_0}_{|(I-P)\X} = I_{|(I-P)\X}$ we also have
\begin{equation}\label{eq:AProp1}
A - AT_0A = 0 .
\end{equation}
Now using (\ref{eq:AProp}) and (\ref{eq:AProp1}) it is easy to see that 

\begin{eqnarray*}
(A+P)G &=& (A+P)(I + T_0(I-P) - T_0A) \\
&=& A + AT_0(I-P) - AT_0A + P = I-P + P = I\\
G(A+P) &=& (I + T_0(I-P) - T_0A)(A+P) \\
&=& A + P + T_0(I-P)A + T_0(I-P)P - T_0AA - T_0AP  \\
&=& A + P + T_0A - T_0AA- T_0AP \\
&=& P + (I - T_0A)A + T_0A(I-P)  \\
&=& P +(I-T_0A)(I-P)A +(I-P)\\
&=& I + ((I-P)-T_0A(I-P))A \\
&=& I + (I-P - (I-P))A = I .
\end{eqnarray*}
Using a similarity we obtain 
$$
\left( \begin{array}{cc}
I & 0   \\
0 & G^{-1} 
\end{array} \right)
\left( \begin{array}{cc}
* & \LSo T\RSt   \\
* & * 
\end{array} \right)
\left( \begin{array}{cc}
I & 0   \\
0 & G 
\end{array} \right) = 
\left( \begin{array}{cc}
* & \LSo T\RSt G   \\
* & * 
\end{array} \right) .
$$
It is clear that we will be done if we show that $\LSo=\LSo T\RSt G$. In order to do this consider
the equation $(A+P)G = I\Leftrightarrow AG + PG = I$. Multiplying both sides of the last equation on the left by $\LSo$ 
gives us $\LSo AG + \LSo PG = \LSo$. Using $\LSo P \equiv \LSo P_{X_0} = 0$ we obtain $\LSo AG = \LSo$. Finally, substituting $A$ from (\ref{eq:A})
 in the last equation yields
$$
\LSo = \LSo AG = \LSo P_{Y_0}T\RSt G = \LSo (I-P_{X_0})T\RSt G = \LSo T\RSt G
$$
which finishes the proof. 
\end{proof}
The following theorem was proved in  \cite{Apostol_lp} for $\X = \ell_p$ , but inessential modifications give the result in the general case.

\begin{thm}\label{diagcomm}
Let $\decomp$ be a decomposition of $\X$ and let $L$ be the left shift associated with it. Then the matrix operator 
$$
\left( \begin{array}{cc}
T_1 & L   \\
T_2 & T_3 
\end{array} \right)
$$
acting on $\X\oplus\X$ is a commutator.
\end{thm}
\begin{proof}
Let $\decomp =\{X_i\}$ be the given decomposition. Consider a decomposition $\decomp_1 = \{Y_i\}$ such that 
$\displaystyle Y_0 = \bigoplus_{i=1}^{\infty}X_i$ and $\displaystyle X_0 = \bigoplus_{i=1}^{\infty}Y_i$.
Now  there exists an operator $G$ such that 
$D_{L_{\decomp}}G = \RSo\LSo (T_1+T_3)$. This can be done using Corollary \ref{largekernel}, since 
$\RSo\LSo = I - P_{Y_0} = P_{X_0}$. By making the similarity
$$
\widetilde{T} := \left( \begin{array}{cc}
I & 0   \\
G & I 
\end{array} \right)
\left( \begin{array}{cc}
T_1 & L   \\
T_2 & T_3 
\end{array} \right)
\left( \begin{array}{cc}
I & 0   \\
-G & I 
\end{array} \right)
=
\left( \begin{array}{cc}
T_1 - LG & L   \\
* & T_3 + GL 
\end{array} \right)
$$
 we have $T_1 + T_3 - LG + GL = T_1 + T_3 - D_LG = T_1 + T_3 - \RSo\LSo (T_1+T_3) = P_{Y_0}(T_1+T_3)$.
Using  again Corollary \ref{largekernel} we deduce that $T_1 + T_3 - LG + GL $ is a commutator. Thus  by replacing $T$ by $\widetilde{T}$ 
we can assume that $T_1+T_3$ is a commutator, say $T_1+T_3= AB-BA$ and $\displaystyle \|A\|<\frac{1}{2}$ (this can be done by scaling).
Denote by $M_T$ left multiplication by the operator $T$. Then 
$\displaystyle \|M_RD_A\| < 1$ where $R$ is the right shift associated with $\decomp$. The operator
$T_0 = (M_I - M_RD_A)^{-1}M_R(T_3B - T_2)$ is well defined and it is easy to see that 
$$
\left( \begin{array}{cc}
A & 0   \\
T_3 & A-L 
\end{array} \right)
\left( \begin{array}{cc}
B & I   \\
T_0 & 0 
\end{array} \right)
-
\left( \begin{array}{cc}
B & I   \\
T_0 & 0 
\end{array} \right)
\left( \begin{array}{cc}
A & 0   \\
T_3 & A-L 
\end{array} \right)
=
\left( \begin{array}{cc}
T_1 & L   \\
T_2 & T_3 
\end{array} \right) .
$$
This finishes the proof. 
\end{proof}
\begin{thm}\label{mainthm}
Let $\X = \ell_1$. An operator $T\in\opX$ is a commutator if and only if $\,\,T-\lambda$ is not compact for any $\lambda\neq 0$.
\end{thm}
\begin{proof}
Note first that if $T$ is a commutator, from the remarks we made in the introduction it follows that $T-\lambda$ cannot be compact for any 
$\lambda\neq 0$. For proving the other direction we have to consider two cases:

{\bf{Case I.}} If $T$ is compact operator ($\lambda = 0$), the statement of the theorem follows from Lemma \ref{CompactLemma}.\\
{\bf{Case II.}} If $\,\,T-\lambda$ is not compact for any $\lambda$, then we consider $\les$.
Since $\les $ is a non-empty set, there exists $\lambda\in\les$ such that $\,\,T-\lambda$ is not compact and we are in a position to
apply Lemma \ref{PTPLemma1} for the operator $T-\lambda$.
 Note that the conclusion of Lemma \ref{PTPLemma1} for $\,\,T-\lambda$ implies that the same claim is true for $T$ as well. 
Now we are in position to apply Theorem \ref{mainauxresult} (which we can because of Proposition \ref{auxprop})
 and obtain that $T$ is similar to an operator of the form $\left( \begin{array}{cc}
* & L  \\
* & * 
\end{array} \right)
$.
Finally, we apply Theorem \ref{diagcomm} to complete the proof.
\end{proof}
\section{Commutators on $\ell_{p_1}\oplus\ell_{p_2}\oplus\cdots\oplus\ell_{p_n}$ and $\ell_{\infty}$}

\begin{lem}\label{diaglemma}
Let $X$ and $Y$ be Banach spaces and  $T = \left( \begin{array}{cc}
A & B   \\
C & D 
\end{array} \right)
$
an operator from $X\oplus Y$ into $X\oplus Y$. If $A$ and $D$ are commutators on the corresponding spaces then $T$ is a commutator on $X\oplus Y$.
\end{lem}
\begin{proof}
Let $A = [A_1, A_2]$ and $D = [D_1, D_2]$. Assume without loss of generality
that $\displaystyle \max (\|A_2\|,\|D_2\|)<\frac{1}{4}$. We need to find operators $E_1$ and $E_2$ such that
$$
T = \left( \begin{array}{cc}
A_1 & E_1   \\
E_2 & D_1 
\end{array} \right)
\left( \begin{array}{cc}
A_2 + I & 0   \\
0 & D_2 
\end{array} \right)
-
\left( \begin{array}{cc}
A_2 + I& 0   \\
0 & D_2 
\end{array} \right)
\left( \begin{array}{cc}
A_1 & E_1   \\
E_2 & D_1 
\end{array} \right),
$$
or equivalently, we have to solve the equations
\begin{eqnarray}
B & = & E_1D_2 - (A_2+I)E_1   \label{eq:diaglemma1}\\
C & = & E_2(A_2+I) - D_2E_2   \label{eq:diaglemma2}
\end{eqnarray}
for $E_1$ and $E_2$.
Let $G:L(X,Y)\to L(X,Y)$ be defined by $G(S) = -SA_2+D_2S$. Clearly $\|G\|<1$ by our choice of $A_2$ and $D_2$ and hence 
$I-G$ is invertible. Now it is enough to observe that (\ref{eq:diaglemma2}) is equivalent to $C = (I-G)(E_2)$ which will give us $E_2 = (I-G)^{-1}C$.
Analogously we define $F:L(Y,X)\to L(Y,X)$ by $F(S) = -A_2S + SD_2$ and then (\ref{eq:diaglemma1}) will be  equivalent to 
$-B = (I-F)(E_1)$. Applying the same argument as above we get that $I-F$ is invertible and hence $E_1 = (I-F)^{-1}(-B)$. 
\end{proof}
\begin{thm}\label{sumlp}
Let $\X = \ell_p\oplus\ell_q$ where $1\leq q<p<\infty$ and $T\in\opX$. Let $P_{\ell_p}$ and $P_{\ell_q}$ be the natural projections
from $\X$ onto $\ell_p$ and $\ell_q$ respectively. Then $T$ is a commutator if and only of $P_{\ell_p}TP_{\ell_p}$ and $P_{\ell_q}TP_{\ell_q}$
are commutators as operators acting on $\ell_p$ and $\ell_q$ respectively. 
\end{thm}
\begin{proof}
Throughout the proof we will work with the matrix representation of $T$ as an operator acting on $\X$. Let 
$T = \left( \begin{array}{cc}
A & B   \\
C & D 
\end{array} \right)
$ where $A:\ell_p\to\ell_p, D:\ell_q\to\ell_q, B:\ell_q\to\ell_p, C:\ell_p\to\ell_q$. The well known fact that the operator $C$ is compact
 (\cite[Proposition 2.c.3]{LT}) will play an important role in the proof. 
If $T$ is a commutator, then  $T = [T_1,T_2]$ for some $T_1, T_2\in\opX$. Write 
$T_i = \left( \begin{array}{cc}
A_i & B_i   \\
C_i & D_i 
\end{array} \right)
$ for $i = 1,2$. A simple computation shows that  
$$
T = \left( \begin{array}{cc}
[A_1, A_2] + B_1C_2 - B_2C_1 & A_1B_2 + B_1D_2 - A_2B_1 - B_2D_1   \\
C_1A_2 + D_1C_2 - C_2A_1 - D_2C_1 & [D_1, D_2] + C_1B_2 - C_2B_1 
\end{array} \right) .
$$
From the classification of the commutators on $\ell_p$ for $1\leq p<\infty$ and the fact that the $C_i$'s are compact we immediately deduce that
the diagonal entries in the last representation of $T$ are commutators. For the other direction we apply  Lemma \ref{diaglemma} which concludes the proof.

\end{proof}
The classification given in the theorem can be immediately generalized to a space which is finite sum of $\ell_p$ spaces, namely, we have the following

\begin{cor}
Let $\X = \ell_{p_1}\oplus\ell_{p_2}\oplus\cdots\oplus\ell_{p_n}$ where $1\leq p_n < p_{n-1}< \ldots <p_1 < \infty $ and $T\in\opX$. Let $P_{\ell_{p_i}}$ 
be  the natural projections from $\X$ onto $\ell_{p_i}$ for $i = 1,2,\ldots ,n$. Then $T$ is a commutator if and only if for each $1\leq i\leq n$,
$P_{\ell_{p_i}}TP_{\ell_{p_i}}$  is a commutator as an operator acting on $\ell_{p_i}$. 
\end{cor}
\begin{proof}
We will proceed by induction on $n$ and clearly Theorem \ref{sumlp} gives us the result for $n=2$. If the statement is true for some $n$, then to show
it for $n+1$, denote $Y = \ell_{p_2}\oplus\ell_{p_3}\oplus\cdots\oplus\ell_{p_n}$. Now $\X = \ell_{p_1}\oplus Y$ and using the same argument as in 
Theorem \ref{sumlp} we can see that if $T$ is a commutator, then both $P_{\ell_{p_1}}TP_{\ell_{p_1}}$ and $P_{\ell_{Y}}TP_{\ell_{Y}}$
are commutators on $\ell_{p_1}$ and $Y$ respectively. Here we use the induction step to show that compact perturbation of a commutator on $Y$ is still a  
commutator. The other direction is exactly as in Theorem \ref{sumlp}. It is worthwhile noticing that for this direction we do not need any assumption on
the spaces in the sum.
\end{proof}

Our last result shows that
every strictly singular operator
in $L(\linf)$ is a commutator. Clearly this is an essential step in proving the conjecture about the classification of the commutators on $\linf$, namely,
that an operator $T\in L(\linf)$ is not a commutator if and only if $T = \lambda I + S$ for some strictly singular operator $S$
and some $\lambda\neq 0$, but because
of the structure of $\linf$ we cannot apply the method developed in this paper. Note also that the ideal of the strictly singular operators 
is the largest ideal in $L(\linf)$ (follows from \cite[Theorem 1.2]{Whitley} and \cite[Corollary 1.4]{Rosenthal_linfty}), 
the proof of which we include for completeness. In order to develop (if at all possible) a similar approach,
one may have to find a suitable substitution for the set $\AD$ defined in (\ref{eq:ADDef}) and an analog of the left essential spectrum 
(Definition \ref{def:les}). Also, a couple of times in this paper we have used the fact that every infinite dimensional subspace of 
$\ell_p$ ($1\leq p < \infty $) contains a further subspace isomorphic to $\ell_p$ and complemented in $\ell_p$, which does not hold for 
$\ell_{\infty}$. This additional obstacle should be overcome as well. First we will prove

\begin{lem}
The ideal of strictly singular operators is the largest ideal in $L(\linf)$.
\end{lem}
\begin{proof}
Assume that $T$ is not a strictly singular operator. Our goal will be to prove that any ideal that contains $T$ must coincide with $L(\linf)$.
Note first that on $\linf$ the ideals of the weakly compact and the strictly singular operators coincide  (\cite[Theorem 1.2]{Whitley}).
Then we use the fact that any non-weakly compact operator is an isomorphism on some subspace $Y$ of $\linf$ isomorphic to $\linf$ 
(\cite[Corollary 1.4]{Rosenthal_linfty}). The subspaces $Y$  and $TY$ will be  automatically complemented in $\linf$ because $\linf$ is an injective space. 
This automatically yields that $I_{\linf}$ factors through $T$ and hence any ideal containing $T$ coincides with $L(\linf)$.
\end{proof}

\begin{thm}
Let $T\in L(\linf)$ be a strictly singular operator.  Then T is a commutator.
\end{thm}
\begin{proof}
Since T is a strictly singular operator, $T$ is weakly compact (\cite[Corollary 1.4]{Rosenthal_linfty} ). 
Thus it follows that $T\linf$ is separable (since any weakly compact subset of the dual to any separable space is metrizable) and let $Y =\overline{T\linf}$.
The space $\linf/Y$ must be non-reflexive since assuming otherwise gives us that Y has a subspace isomorphic to $\linf$ 
(\cite[Theorem 4]{LR_automorphisms}).  Now consider the quotient map $Q: \linf\to\linf/Y$. $Q$ is not weakly compact and hence 
(using again \cite[Corollary 1.4 ]{Rosenthal_linfty}) there exists $X\simeq \linf, X\subset\linf$ such that $Q_{|X}$ is an isomorphism.
Let $P'$ be a projection onto $QX$ and set $P = (Q_{|X})^{-1}P'Q$. $P$ is a projection in $\linf,\, PY=\{0\}$ and by the construction, $P\linf$ 
isomorphic to $\linf$. Thus it follows that $PT = 0$ and we obtain
that $T$ is similar to an operator $T'$ for which there exists a $M\subset\N$ such that $P_MT' = 0$. Using \cite[Theorem 2.9 ]{Apostol_lp}
we conclude that $T'$ is commutator and hence $T$ is commutator.
\end{proof}

\newpage

Department of Mathematics, Texas A\&M University, College Station, TX 77840\\
\phantom{qqm}{\textit{E-mail address :}} {\textbf{dossev@math.tamu.edu}}
\end{document}